\documentclass[11pt]{amsart}

\usepackage{geometry}
\usepackage[british]{babel}
\usepackage{mathtools}
\usepackage{amsmath}
\usepackage{amssymb}
\usepackage{amsthm}
\usepackage{amstext}
\usepackage{thmtools}
\usepackage{enumerate}
\usepackage[shortlabels]{enumitem}
\usepackage[margin=1cm]{caption}
\usepackage{subcaption}
\usepackage{float}
\usepackage[hidelinks]{hyperref}
\usepackage[capitalise]{cleveref}
\usepackage[sort=use]{glossaries}
\usepackage{graphicx}
\usepackage{tikz}
\usepackage{circuitikz}
\usepackage{pgfplots}
\usepackage{etoolbox}
\usepackage[T1]{fontenc}

\theoremstyle{definition}
\newtheorem{definition}{Definition}[section]

\theoremstyle{plain}
\newtheorem{theorem}[definition]{Theorem}

\newtheorem{question}[definition]{Question}

\newtheorem{lemma}[definition]{Lemma}

\Crefname{fact}{Fact}{Facts}
\Crefname{claim}{Claim}{Claims}

\numberwithin{equation}{section}

\newtoggle{inclaimproof}
\togglefalse{inclaimproof}

\AtBeginEnvironment{proof}{%
  \iftoggle{inclaimproof}{}{%
    \setcounter{claim}{0}%
  }%
}

\newcommand{\cF}{\mathcal{F}}
\newcommand{\cH}{\mathcal{H}}

\newcommand{\bN}{\mathbb{N}}

\tikzset{vtx/.style={inner sep=1.7pt, outer sep=0pt, circle, fill,draw}}
\pgfplotsset{compat=1.18}

\title{2-Colouring shift-chains}
\author{Zak Smith}
\thanks{The research leading to these results was partially supported by the Deutsche Forschungsgemeinschaft (DFG, German Research Foundation) -- 428212407.}

\begin{document}
\begin{abstract}
    A \emph{shift-chain} is an $ r $-uniform hypergraph $ \cH $ on vertex set $ [n] $ with the property that, for any two edges $ \{ e_1, \ldots, e_r \} $ and $ \{ f_1, \ldots, f_r \} $ with $ e_1 < \cdots < e_r $ and $ f_1 < \cdots < f_r $, either $ e_i \le f_i $ for all $ i \in [r] $ or $ f_i \le e_i $ for all $ i \in [r] $.
    It is known that all shift-chains are properly vertex-colourable with three colours (that is, such that no edge is monochromatic), which is optimal for $ r \in \{ 2, 3 \} $.
    It was asked by Pálvölgyi in 2010 whether all shift-chains of sufficiently large uniformity are properly $ 2 $-colourable.
    We answer this question in a strong form, proving that in fact all shift-chains of uniformity at least $ 4 $ are properly $ 2 $-colourable.
    The colouring is obtained via a natural algorithm with linear running time.
\end{abstract}

\maketitle

\section{Introduction} \label{sect:intro}

The study of colourings in graphs, and more generally hypergraphs, is one of the most classical problems in combinatorics and has been the subject of extensive research for over a century.
We say that a hypergraph is \emph{properly $ k $-colourable} if there exists a partition of the vertex set into $ k $ parts such that no part contains an edge.
In particular, the property of being properly $ 2 $-colourable is sometimes referred to as \emph{property B}, owing to its introduction by Bernstein~\cite{Bernstein1908TrigonometrischeReihen} in 1908, and has been very extensively studied in various classes of hypergraphs (see for example~\cite{DamasdiEtAl2025Survey,GrillLinzmayer2026,RaigorodskiiShabanov2011}).
While a simple linear-time algorithm can determine whether a graph is bipartite, the general problem of determining whether a hypergraph has property B is NP-complete~\cite{Lovasz1973Coverings}, so a major line of research is to seek natural classes of hypergraphs for which this property is satisfied.

We consider a class of $ r $-uniform hypergraphs $ \cH $ first proposed by Pálvölgyi~\cite{Palvolgyi2010Decomposition} in the context of problems in discrete geometry.
We suppose that the vertex set $ V(\cH) $ has a total order $ <_V $ and think of each edge $ \vec{e} \in E(\cH) $ as equivalent to the (unique) vector $ (e_1, \ldots, e_r) \in V(\cH)^r $ for which $ \vec{e} = \{ e_1, \ldots, e_r \} $ and $ e_1 < \cdots < e_r $.
We may thus define a partial order $ \le_E $ on the edge set $ E(\cH) $ by writing $ \vec{e} \le \vec{f} $ if $ e_i \le f_i $ for every $ i \in [r] $, and write further $ \vec{e} < \vec{f} $ if additionally $ \vec{e} \ne \vec{f} $.
Say that $ \vec{e} $ and $ \vec{f} $ are \emph{comparable} if either $ \vec{e} < \vec{f} $ or $ \vec{f} < \vec{e} $.
We say that $ \cH $ is an \emph{$ r $-shift-chain} (or just \emph{shift-chain}) if $ <_E $ is a total order on $ E(\cH) $, that is, if every pair of distinct edges is comparable.

\begin{figure}[ht]
    \centering
    \begin{subfigure}[c]{0.6\textwidth}
        \centering
        \begin{tikzpicture}[scale=0.5]
            \node at (0, 1) {$1$};
            \node at (1, 1) {$2$};
            \node at (2, 1) {$3$};
            \node at (3, 1) {$4$};
            \node at (4, 1) {$5$};
            \node at (5, 1) {$6$};
        
            \node at (0, 0) [vtx] {};
            \node at (1, 0) [vtx] {};
            \node at (2, 0) [vtx] {};
            \node at (3, 0) [vtx,color=gray] {};
    
            \node at (0, -1) [vtx] {};
            \node at (1, -1) [vtx] {};
            \node at (2, -1) [vtx] {};
            \node at (4, -1) [vtx,color=gray] {};
    
            \node at (0, -2) [vtx] {};
            \node at (2, -2) [vtx] {};
            \node at (3, -2) [vtx] {};
            \node at (4, -2) [vtx,color=gray] {};
    
            \node at (1, -3) [vtx] {};
            \node at (2, -3) [vtx] {};
            \node at (3, -3) [vtx] {};
            \node at (5, -3) [vtx,color=gray] {};

            \node at (2, -4) [vtx] {};
            \node at (3, -4) [vtx] {};
            \node at (4, -4) [vtx] {};
            \node at (5, -4) [vtx,color=gray] {};
        \end{tikzpicture}
        \caption{A $ 4 $-shift-chain on $ n = 6 $ vertices with five edges. The black vertices also form a $ 3 $-shift-chain with four (distinct) edges.}
    \end{subfigure}
    \hfill
    \begin{subfigure}[c]{0.35\textwidth}
        \centering
        \begin{tikzpicture}[scale=0.5]
            \node at (0, 1) {$1$};
            \node at (1, 1) {$2$};
            \node at (2, 1) {$3$};
            \node at (3, 1) {$4$};
            \node at (4, 1) {$5$};
            \node at (5, 1) {$6$};
            \node at (6, 1) {$7$};
            
            \node at (0, 0) [vtx] {};
            \node at (1, 0) [vtx] {};
            \node at (4, 0) [vtx] {};
            \node at (6, 0) [vtx] {};
    
            \node at (2, -1) [vtx] {};
            \node at (3, -1) [vtx] {};
            \node at (4, -1) [vtx] {};
            \node at (5, -1) [vtx] {};

            \node at (-1, 0) {$\vec{e}$};
            \node at (-1, -1) {$\vec{f}$};
        \end{tikzpicture}
        \caption{The edges $ \vec{e} $ and $ \vec{f} $ are not comparable because $ e_1 < f_1 $ but $ f_4 < e_4 $.} 
    \end{subfigure}

    \caption{A visual representation of shift-chains; the columns correspond to vertices and the rows to edges.} \label{fig:shift_chain_def}
\end{figure}
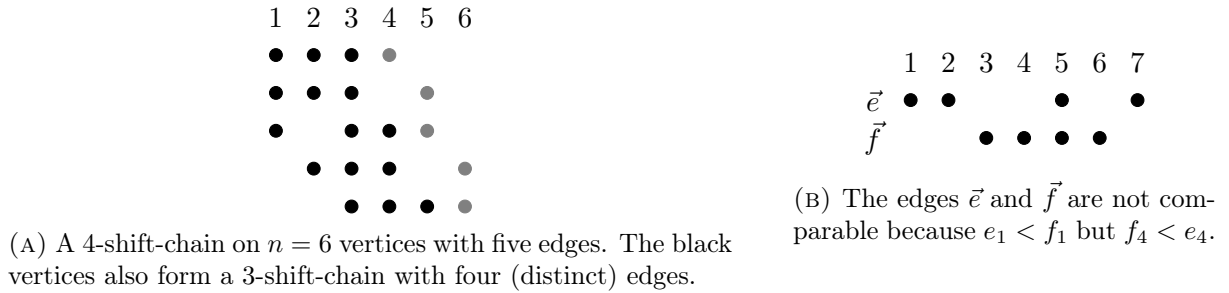

In 2010, Pálvölgyi~\cite{Palvolgyi2010Decomposition} posed the following question as a first step towards tackling a conjecture of Pach~\cite{Pach1980Decomposition} from 1980, asserting that all planar convex sets are cover-decomposable.

\begin{question}[Question 9.5 in~\cite{Palvolgyi2010Decomposition}] \label{conj:shift_chain_2col}
    Does there exist $ r_0 \in \bN $ such that, for every $ r \ge r_0 $, all $ r $-shift-chains are properly $ 2 $-colourable?
\end{question}

The cover-decomposability conjecture itself was disproven in a major breakthrough by Pach and Pálvölgyi~\cite{PachPalvolgyi2016Unsplittable} in 2013 but progress on \cref{conj:shift_chain_2col} has been essentially limited to a few simple observations (see~\cref{sect:shift_chain_obs}), despite considerable interest in the past decade~\cite{BosekEtAl2020Coloring,DamasdiEtAl2025Survey,PachPalvolgyi2016Unsplittable,PachPalvolgyiToth2013Survey,Ueckerdt2024ShiftChains}.
Other colouring problems for shift-chains have also been studied, with some progress~\cite{BosekEtAl2020Coloring,Ueckerdt2024ShiftChains}.
In particular, Pach~and~Pálvölgyi~\cite{PachPalvolgyi2016Unsplittable} proved that so-called `special' shift-chains satisfying an additional condition are properly~$ 2 $-colourable, which they used to deduce a version of the cover-decomposability conjecture for unbounded open convex sets.
In their paper~\cite{PachPalvolgyi2016Unsplittable}, they reiterated the question of whether this property is satisfied by shift-chains in general.
Our contribution is to answer \cref{conj:shift_chain_2col} affirmatively, showing that we can in fact take $ r_0 = 4 $.

\begin{theorem} \label{thm:shift_chain_2col}
    All shift-chains of uniformity at least $ 4 $ are properly $ 2 $-colourable.
\end{theorem}

Indeed, this choice of $ r_0 $ is the best possible, since there are known examples of $ 2 $- and $ 3 $-shift-chains which are not properly $ 2 $-colourable (see~\cref{sect:shift_chain_obs}).
We remark also that the algorithm we provide can be implemented to run in linear time.

The rest of this section contains a brief outline of our algorithm and a summary of the known properties of shift-chains.
A complete formulation of the algorithm is then given in \cref{sect:algorithm} and we prove its correctness in \cref{sect:proof}, finishing in \cref{sect:linear} with a brief justification of its linear time complexity.

\subsection{Algorithm outline} \label{sect:alg_outline}

Our algorithm consists of three stages, as depicted in \Cref{fig:algo_def}.
We start in Stage 1 with the obvious greedy algorithm, colouring the vertices $ 1, \ldots, n $ one at a time in blocks of alternating colours, always giving a vertex the same colour as the previous vertex unless this would cause an edge to be fully contained in the current monochromatic block, in which case we switch colours.
We show that any monochromatic edge $ e $ created by this greedy colouring has its first three vertices in some monochromatic block $ B $, no vertices in the next block $ C $, and its final vertex in the following block $ D $ (see \Cref{fig:algo_def_stage1}).
Thus, in Stage~2, we process such monochromatic edges $ e $ from right to left, pushing the left boundary of the central block $ C $ backwards, so that the third vertex of $ e $ is contained in $ C $.
However, this may cause another edge $ f $ to become monochromatic as follows (see \Cref{fig:algo_def_stage2}): suppose the first two vertices of $ f $ lie in the block $ A $ preceding $ B $, and after Stage~1 the third vertex was in block $ B $ and the fourth in block $ C $, but in Stage~2 the third vertex is absorbed into the block $ C $, which receives the same colour as $ A $.
In fact, we show that this is the only type of monochromatic edge left after the second stage, and process these from left to right in Stage~3, splitting the block $ C $ after the third vertex of $ f $ and flipping the colours of subsequent blocks (see \Cref{fig:algo_def_stage3}), which we prove does not create any new monochromatic edges.

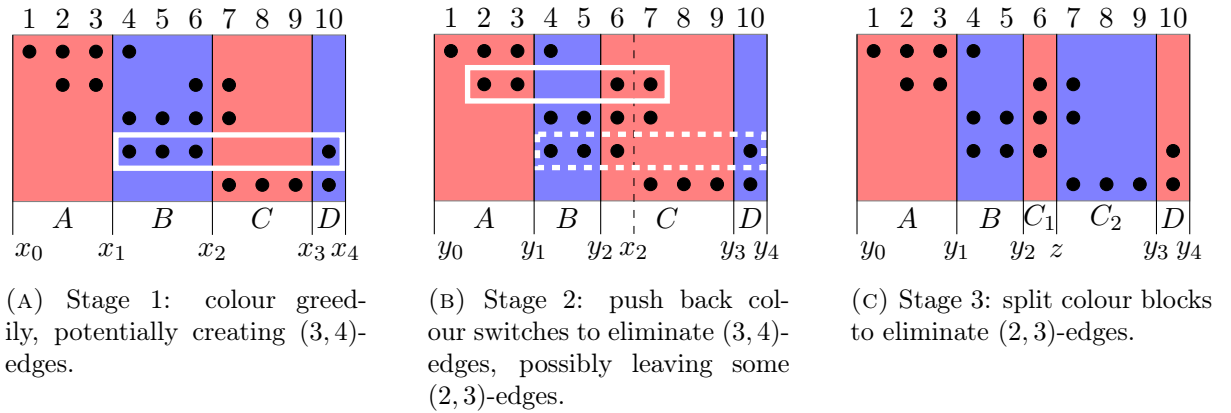
\begin{figure}[ht]
    \centering
    \begin{subfigure}[t]{0.3\textwidth}
        \centering
        \begin{tikzpicture}[scale=0.44]
            \node at (0, 1) {$1$};
            \node at (1, 1) {$2$};
            \node at (2, 1) {$3$};
            \node at (3, 1) {$4$};
            \node at (4, 1) {$5$};
            \node at (5, 1) {$6$};
            \node at (6, 1) {$7$};
            \node at (7, 1) {$8$};
            \node at (8, 1) {$9$};
            \node at (9, 1) {$10$};

            \draw[fill=red,opacity=0.5] (-0.5,0.5) rectangle ++(3,-5);
            \draw[fill=blue,opacity=0.5] (2.5,0.5) rectangle ++(3,-5);
            \draw[fill=red,opacity=0.5] (5.5,0.5) rectangle ++(3,-5);
            \draw[fill=blue,opacity=0.5] (8.5,0.5) rectangle ++(1,-5);

            \draw (-0.5,0.5) -- (-0.5,-5.5);
            \draw (2.5,0.5) -- (2.5,-5.5);
            \draw (5.5,0.5) -- (5.5,-5.5);
            \draw (8.5,0.5) -- (8.5,-5.5);
            \draw (9.5,0.5) -- (9.5,-5.5);
        
            \node at (0, 0) [vtx] {};
            \node at (1, 0) [vtx] {};
            \node at (2, 0) [vtx] {};
            \node at (3, 0) [vtx] {};
    
            \node at (1, -1) [vtx] {};
            \node at (2, -1) [vtx] {};
            \node at (5, -1) [vtx] {};
            \node at (6, -1) [vtx] {};
    
            \node at (3, -2) [vtx] {};
            \node at (4, -2) [vtx] {};
            \node at (5, -2) [vtx] {};
            \node at (6, -2) [vtx] {};
    
            \node at (3, -3) [vtx] {};
            \node at (4, -3) [vtx] {};
            \node at (5, -3) [vtx] {};
            \node at (9, -3) [vtx] {};

            \node at (6, -4) [vtx] {};
            \node at (7, -4) [vtx] {};
            \node at (8, -4) [vtx] {};
            \node at (9, -4) [vtx] {};

            \node at (0,-6) {$x_0$};
            \node at (2.5,-6) {$x_1$};
            \node at (5.5,-6) {$x_2$};
            \node at (8.5,-6) {$x_3$};
            \node at (9.5,-6) {$x_4$};

            \node at (1,-5) {$A$};
            \node at (4,-5) {$B$};
            \node at (7,-5) {$C$};
            \node at (9,-5) {$D$};

            \draw[draw=white,line width=2pt] (2.6,-2.5) rectangle ++(6.8,-1);
        \end{tikzpicture}
        \caption{Stage 1: colour greedily, potentially creating $(3,4)$-edges.} \label{fig:algo_def_stage1}
    \end{subfigure}
    \hfill
    \begin{subfigure}[t]{0.3\textwidth}
        \centering
        \begin{tikzpicture}[scale=0.44]
            \node at (0, 1) {$1$};
            \node at (1, 1) {$2$};
            \node at (2, 1) {$3$};
            \node at (3, 1) {$4$};
            \node at (4, 1) {$5$};
            \node at (5, 1) {$6$};
            \node at (6, 1) {$7$};
            \node at (7, 1) {$8$};
            \node at (8, 1) {$9$};
            \node at (9, 1) {$10$};

            \draw[fill=red,opacity=0.5] (-0.5,0.5) rectangle ++(3,-5);
            \draw[fill=blue,opacity=0.5] (2.5,0.5) rectangle ++(2,-5);
            \draw[fill=red,opacity=0.5] (4.5,0.5) rectangle ++(4,-5);
            \draw[fill=blue,opacity=0.5] (8.5,0.5) rectangle ++(1,-5);

            \draw (-0.5,0.5) -- (-0.5,-5.5);
            \draw (2.5,0.5) -- (2.5,-5.5);
            \draw (4.5,0.5) -- (4.5,-5.5);
            \draw[dashed] (5.5,0.5) -- (5.5,-5.5);
            \draw (8.5,0.5) -- (8.5,-5.5);
            \draw (9.5,0.5) -- (9.5,-5.5);
        
            \node at (0, 0) [vtx] {};
            \node at (1, 0) [vtx] {};
            \node at (2, 0) [vtx] {};
            \node at (3, 0) [vtx] {};
    
            \node at (1, -1) [vtx] {};
            \node at (2, -1) [vtx] {};
            \node at (5, -1) [vtx] {};
            \node at (6, -1) [vtx] {};
    
            \node at (3, -2) [vtx] {};
            \node at (4, -2) [vtx] {};
            \node at (5, -2) [vtx] {};
            \node at (6, -2) [vtx] {};
    
            \node at (3, -3) [vtx] {};
            \node at (4, -3) [vtx] {};
            \node at (5, -3) [vtx] {};
            \node at (9, -3) [vtx] {};

            \node at (6, -4) [vtx] {};
            \node at (7, -4) [vtx] {};
            \node at (8, -4) [vtx] {};
            \node at (9, -4) [vtx] {};

            \node at (0,-6) {$y_0$};
            \node at (2.5,-6) {$y_1$};
            \node at (4.5,-6) {$y_2$};
            \node at (5.5,-6) {$x_2$};
            \node at (8.5,-6) {$y_3$};
            \node at (9.5,-6) {$y_4$};

            \node at (1,-5) {$A$};
            \node at (3.5,-5) {$B$};
            \node at (6.5,-5) {$C$};
            \node at (9,-5) {$D$};

            \draw[draw=white,line width=2pt] (0.5,-0.5) rectangle ++(6,-1);
            \draw[draw=white,dashed,line width=2pt] (2.6,-2.5) rectangle ++(6.8,-1);
        \end{tikzpicture}
        \caption{Stage 2: push back colour switches to eliminate $(3,4)$-edges, possibly leaving some $(2,3)$-edges.} \label{fig:algo_def_stage2}
    \end{subfigure}
    \hfill
    \begin{subfigure}[t]{0.3\textwidth}
        \centering
        \begin{tikzpicture}[scale=0.44]
            \node at (0, 1) {$1$};
            \node at (1, 1) {$2$};
            \node at (2, 1) {$3$};
            \node at (3, 1) {$4$};
            \node at (4, 1) {$5$};
            \node at (5, 1) {$6$};
            \node at (6, 1) {$7$};
            \node at (7, 1) {$8$};
            \node at (8, 1) {$9$};
            \node at (9, 1) {$10$};

            \draw[fill=red,opacity=0.5] (-0.5,0.5) rectangle ++(3,-5);
            \draw[fill=blue,opacity=0.5] (2.5,0.5) rectangle ++(2,-5);
            \draw[fill=red,opacity=0.5] (4.5,0.5) rectangle ++(1,-5);
            \draw[fill=blue,opacity=0.5] (5.5,0.5) rectangle ++(3,-5);
            \draw[fill=red,opacity=0.5] (8.5,0.5) rectangle ++(1,-5);

            \draw (-0.5,0.5) -- (-0.5,-5.5);
            \draw (2.5,0.5) -- (2.5,-5.5);
            \draw (4.5,0.5) -- (4.5,-5.5);
            \draw (5.5,0.5) -- (5.5,-5.5);
            \draw (8.5,0.5) -- (8.5,-5.5);
            \draw (9.5,0.5) -- (9.5,-5.5);
        
            \node at (0, 0) [vtx] {};
            \node at (1, 0) [vtx] {};
            \node at (2, 0) [vtx] {};
            \node at (3, 0) [vtx] {};
    
            \node at (1, -1) [vtx] {};
            \node at (2, -1) [vtx] {};
            \node at (5, -1) [vtx] {};
            \node at (6, -1) [vtx] {};
    
            \node at (3, -2) [vtx] {};
            \node at (4, -2) [vtx] {};
            \node at (5, -2) [vtx] {};
            \node at (6, -2) [vtx] {};
    
            \node at (3, -3) [vtx] {};
            \node at (4, -3) [vtx] {};
            \node at (5, -3) [vtx] {};
            \node at (9, -3) [vtx] {};

            \node at (6, -4) [vtx] {};
            \node at (7, -4) [vtx] {};
            \node at (8, -4) [vtx] {};
            \node at (9, -4) [vtx] {};

            \node at (0,-6) {$y_0$};
            \node at (2.5,-6) {$y_1$};
            \node at (4.5,-6) {$y_2$};
            \node at (5.5,-6) {$z$};
            \node at (8.5,-6) {$y_3$};
            \node at (9.5,-6) {$y_4$};

            \node at (1,-5) {$A$};
            \node at (3.5,-5) {$B$};
            \node at (5,-5) {$C_1$};
            \node at (7,-5) {$C_2$};
            \node at (9,-5) {$D$};
        \end{tikzpicture}
        \caption{Stage 3: split colour blocks to eliminate $(2,3)$-edges.} \label{fig:algo_def_stage3}
    \end{subfigure}
    \caption{The three stages of the algorithm defined in \cref{sect:algorithm}; the white rectangles indicate monochromatic edges after each of the first two stages.} \label{fig:algo_def}
\end{figure}

\subsection{Observations about shift-chains} \label{sect:shift_chain_obs}

We now briefly summarise what is already known about shift-chains.
This natural class of hypergraphs has a neat visual representation (depicted in~\Cref{fig:shift_chain_def}) and various nice properties.
Observe for example that an $ r $-shift-chain $ \cH $ has at most $ r(n - r) + 1 $ edges~\cite{PachPalvolgyiToth2013Survey}.
Indeed, writing $ T \coloneqq e(\cH) $, we may express the edge set $ E(\cH) = \{ \vec{e}(t): t \in [T] \} $ as a sequence of vectors $ (e_1(t), \ldots, e_r(t))_{t \in T} $, intuitively thinking of this as a single vector of $ r $ vertices which are shifted over time, always shifting from left to right.
Between adjacent time steps $ t $ and $ t + 1 $ with $ t \in [T - 1] $, at least one of the $ r $ values $ e_i(t) $ increases, and the others do not decrease.
Since $ i \le e_i(t) \le n - r + i $ for each $ i \in [r] $ and $ t \in [T] $, each value may increase at most $ n - r $ times, so it follows that $ T - 1 \le r (n - r) $.
Another way to see this is to observe that the sums of the vertices in each edge are pairwise distinct and lie in the interval $ [\frac{r(r + 1)}{2}, rn - \frac{r(r-1)}{2}] $.

It is perhaps not immediately obvious why, as in the phrasing of \cref{conj:shift_chain_2col}, we might expect shift-chains of sufficiently large uniformity to be properly $ 2 $-colourable, even if this fails for lower uniformities.
Indeed, this is due to the following simple observation~\cite{BosekEtAl2020Coloring}.
Given $ 2 \le r' < r $ and $ I \subseteq [r] $ of size $ |I| = r' $, any $ r $-shift-chain $ \cH $ induces an $ r' $-shift-chain $ \cH' $ on the same vertex set by taking $ E(\cH') \coloneqq \{ (e_i)_{i \in I}: \vec{e} \in E(\cH) \} $; see \Cref{fig:shift_chain_def} for an example with $ r = 4 $, $ r' = 3 $, and $ I = \{ 1, 2, 3 \} $.
Since any proper $ k $-colouring of $ \cH' $ is clearly also a proper $ k $-colouring of $ \cH $, this means that, if we can prove that every $ r_0 $-shift-chain is $ k $-colourable for some $ r_0, k \ge 2 $, then it follows immediately that every $ r $-shift-chain is also $ k $-colourable, for any $ r \ge r_0 $.

In fact, it was shown in~\cite{BosekEtAl2020Coloring} that all $ 2 $-shift-chains (and therefore all shift-chains of any uniformity) are properly $ 3 $-colourable using a simple greedy algorithm.
We omit the details of their argument, remarking instead that the first stage of our algorithm yields a very simple alternative greedy algorithm for a proper $ 3 $-colouring.
We outline the details after the proof of \cref{lem:inv_greedy} in \cref{sect:proof}.
This leaves the obvious question as to whether the bound $ \chi(\cH) \le 3 $ for shift-chains $ \cH $ is tight in general.

Indeed, for $ r = 2 $, observe that the triangle $ \{ 12, 13, 23 \} $ is a $2$-shift-chain with chromatic number~$ 3 $.
Likewise, for $ r = 3 $, examples on nine vertices show that three colours are sometimes necessary.
An example with 19 edges was originally found by Fulek~\cite{Palvolgyi2010Decomposition} using a computer search; see \Cref{fig:counterexample_r3} for one we found with 11 edges.
As such, we see that the bound $ r \ge 4 $ in \cref{thm:shift_chain_2col} is best possible, and indeed it suffices to prove \cref{thm:shift_chain_2col} for $ 4 $-shift-chains.

\begin{figure}[ht]
    \centering
    \begin{tikzpicture}[scale=0.5]
        \node at (0, 1) {$1$};
        \node at (1, 1) {$2$};
        \node at (2, 1) {$3$};
        \node at (3, 1) {$4$};
        \node at (4, 1) {$5$};
        \node at (5, 1) {$6$};
        \node at (6, 1) {$7$};
        \node at (7, 1) {$8$};
        \node at (8, 1) {$9$};

        \def\vsep{0.7};
    
        \node at (0, 0) [vtx] {};
        \node at (1, 0) [vtx] {};
        \node at (2, 0) [vtx] {};
    
        \node at (0, -\vsep) [vtx] {};
        \node at (3, -\vsep) [vtx] {};
        \node at (4, -\vsep) [vtx] {};
    
        \node at (1, -2*\vsep) [vtx] {};
        \node at (3, -2*\vsep) [vtx] {};
        \node at (4, -2*\vsep) [vtx] {};
    
        \node at (2, -3*\vsep) [vtx] {};
        \node at (3, -3*\vsep) [vtx] {};
        \node at (4, -3*\vsep) [vtx] {};
    
        \node at (2, -4*\vsep) [vtx,color=gray] {};
        \node at (3, -4*\vsep) [vtx,color=gray] {};
        \node at (5, -4*\vsep) [vtx,color=gray] {};
    
    
        \node at (2, -5*\vsep) [vtx,color=gray] {};
        \node at (4, -5*\vsep) [vtx,color=gray] {};
        \node at (6, -5*\vsep) [vtx,color=gray] {};
    
    
        \node at (3, -6*\vsep) [vtx,color=gray] {};
        \node at (5, -6*\vsep) [vtx,color=gray] {};
        \node at (6, -6*\vsep) [vtx,color=gray] {};
    
        \node at (4, -7*\vsep) [vtx] {};
        \node at (5, -7*\vsep) [vtx] {};
        \node at (6, -7*\vsep) [vtx] {};
    
        \node at (4, -8*\vsep) [vtx] {};
        \node at (5, -8*\vsep) [vtx] {};
        \node at (7, -8*\vsep) [vtx] {};
    
        \node at (4, -9*\vsep) [vtx] {};
        \node at (5, -9*\vsep) [vtx] {};
        \node at (8, -9*\vsep) [vtx] {};
    
        \node at (6, -10*\vsep) [vtx] {};
        \node at (7, -10*\vsep) [vtx] {};
        \node at (8, -10*\vsep) [vtx] {};
    \end{tikzpicture}
    \caption{A $ 3 $-shift-chain which is not properly $ 2 $-colourable. Indeed, observe that the first four edges ensure that $ 4 $ and $ 5 $ must receive different colours, while the last four edges ensure that $ 5 $ and $ 6 $ must receive different colours. Supposing without loss of generality that $ 4 $ and $ 6 $ are coloured red, while $ 5 $ is coloured blue, it is not possible to colour $ 3 $ and $ 7 $ without making one of the three grey edges monochromatic.} \label{fig:counterexample_r3}
\end{figure}
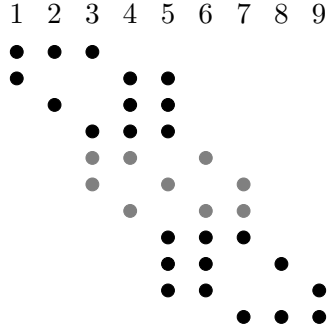

\subsection{Notation}

Given integers $ a, b $ we write $ [a, b] \coloneqq \{ a, a + 1, \ldots, b - 1, b \} $, taking $ [a, b] \coloneqq \emptyset $ for $ b < a $, and in particular $ [a] \coloneqq [1, a] $.
For the remainder of this paper, we fix $ r = 4 $ and let $ \cH $ be a fixed $ 4 $-shift-chain, taking $ V(\cH) = [n] $ with the usual ordering for some $ n \ge 4 $, without loss of generality.

\subsection*{AI declaration}

The colouring algorithm presented in this paper, as well as a sketched proof of its correctness, were generated by a query to ChatGPT 5.6 Sol.
The proof and exposition we present, while drawing on the ideas given by the ChatGPT, were produced and verified by the author.
ChatGPT was also used while proofreading this paper to assist in finding typographical errors. 
However, none of the text present in the final paper is originally due to generative AI, and the author takes full accountability for both the text and mathematical content of this paper.

\subsection*{Acknowledgements}

The research leading to these results was partially supported by the Deutsche Forschungsgemeinschaft (DFG, German Research Foundation) -- 428212407.
This problem was brought to the author's attention by Simona Boyadzhiyska at the workshop “Extremal and Probabilistic Combinatorics” held in July 2025 at the International Centre for Mathematical Sciences (ICMS) in Edinburgh, where we began to work on it together, along with Julia Böttcher and Emma Hogan.
In particular, the author would like to thank Emma Hogan for pointing out a slight simplification of the algorithm presented in an earlier version of this manuscript.
The author would also like to thank Letícia Mattos and Felix Joos for their assistance in proofreading this manuscript.

\section{Our algorithm} \label{sect:algorithm}

In this section, we formally define our three-stage algorithm, which defines a colouring $ c: [n] \to [2] $ via a set of so-called switches.
Recall that it suffices to prove \cref{thm:shift_chain_2col} for shift-chains of uniformity exactly $ 4 $, and that we fixed such a hypergraph $ \cH $.
We begin with some important definitions.

Given $ v \in [0, n] $, say that $ v^+ \coloneqq (v+1)^- \coloneqq (v, v+1) $ is a \emph{wall} and write $ W \coloneqq W(n) $ for the set of all walls.
Say that $ I \subseteq W $ is an \emph{interval} (of walls) if $ I = \{ v^+: v \in [a, b] \} $ for some $ 0 \le a \le b \le n $.
Given a colouring $ c: [n] \to [2] $, define a \emph{block} to be a monochromatic interval (of integers) which is maximal under inclusion, and given two consecutive blocks $ [u, v] $ and $ [v + 1, w] $, we say that the wall $ (v, v + 1) $ is a \emph{switch in $ c $}.
Observe that a colouring $ c $ is determined up to isomorphism by the set $ S(c) $ of all switches in $ c $.
Indeed, given a set $ S $ of walls, we write $ c_S $ for the unique colouring $ c $ with $ c(1) = 1 $ and $ S(c) = S \setminus \{ 1^-, n^+ \} $.
For our purposes, the value of $ c(1) $ is not important, so the goal of our algorithm is to define a set of switches which yields the desired proper $ 2 $-colouring.
We first need some further definitions.

Given a wall $ x = u^+ $ and a vertex $ v \in [n] $, write $ v < x $ if $ v \le u $ and likewise $ x < v $ if $ u + 1 \le v $; this may be uniquely extended to a total order on the union of $ [n] $ with the set of all walls.
Given an edge $ \vec{e} \in E(\cH) $, say that \emph{$\vec{e}$ contains $ x $} if $ e_1 < x < e_4 $; on the other hand, write $ x < \vec{e} $ if $ x < e_1 $ and $ \vec{e} < x $ if $ e_4 < x $.
Say that $ x $ has \emph{position $ (i, i+1) $ in $ \vec{e} $} if $ e_i < x < e_{i + 1} $ for some $ i \in [3] $.
Given also a colouring $ c $, define the \emph{switch set $ S_c(\vec{e}) $ of $ \vec{e} $ in $ c $} to be the set of switches of $ c $ which are contained by $ \vec{e} $, and call $ |S_{c}(\vec{e})| $ the \emph{switch count of $ \vec{e} $ in $ c $}.
Observe that $ \vec{e} $ is monochromatic if and only if there are an even number of switches $ x $ at each of the three positions $ (1,2)$, $(2,3)$, and $(3,4) $.
Given $ k \in \bN $, say that $ c $ is \emph{$ k $-well-spread} if every edge of $ \cH $ has switch count at least 1 and at most $ k $.
Observe that any monochromatic edge in a $ 3 $-well-spread colouring must have exactly two switches and they must have the same position.
As such, we define an \emph{$(i, i+1) $-edge in $ c $} to be a (monochromatic) edge with exactly two switches, both in position $ (i, i+1) $.
Write $ M(c) $ for the set of all edges which are monochromatic in the colouring $ c $.

\textbf{Stage 1: Greedy initial colouring.}
We define an initial set of switches $ X \coloneqq \{ x_1, \ldots, x_s \} \subseteq W $ with $ x_1 < \ldots < x_s $ using the natural greedy algorithm.
Recall our simple idea from \cref{sect:alg_outline}: we colour the first vertex arbitrarily and, for each $ v \in [2, n] $, we give $ v $ the same colour as $ v - 1 $ unless this causes an edge to be entirely contained in the current block, in which case we switch colours.
Formally, we define inductively
\begin{equation} \label{eqn:def_greedy}
    x_0 \coloneqq (0, 1) \quad \text{and} \quad x_{i + 1} \coloneqq e_4^-, \quad \text{where} \quad \vec{e} \coloneqq \vec{e}(i) > x_i \text{ is minimal},
\end{equation}
given $ x_0, \ldots, x_i $ for $ i \ge 0 $.
If no such edge $ \vec{e} $ exists for some $ i \ge 0 $, then the process terminates with $ s \coloneqq i $ and we set $ c_0 \coloneqq c_X $.
Observe that $ 0^+ \le x_i < x_{i+1} \le n^- $ for each $ i \in [0, s - 1] $ by definition, so this process must terminate with $ s \le n - 1 $.
We will show in \cref{lem:inv_greedy} that the resulting colouring $ c_0 $ is $2$-well-spread and all monochromatic edges are $ (3, 4) $-edges, so we aim to handle these in the second stage of our algorithm.

\textbf{Stage 2: Push back switches for $ (3, 4) $-edges.}
We now process any $ (3, 4) $-edges from right to left, pushing the first of the two switches back to the wall immediately before the third vertex of the edge, thus ensuring that the third vertex has a different colour from the others (but potentially creating new monochromatic edges).
Specifically, we define a second set of switches $ Y \coloneqq \{ y_1, \ldots, y_s \} \subseteq W $ inductively by setting
\begin{equation} \label{eqn_def_init_retreat}
    y_i(0) \coloneqq x_i \text{ for each } i \in [s]
\end{equation}
and performing the following for each $ t \ge 0 $.
Suppose we have defined $ Y(t) \coloneqq \{ y_1(t), \ldots, y_s(t) \} $ and that there is at least one $ (3, 4) $-edge in the colouring $ c_t \coloneqq c_{Y(t)} $.
Take $ \vec{e} \coloneqq \vec{e}(t) $ to be the maximal such edge, and let $ i \coloneqq i(t) \in [s - 1] $ be such that
$$ y_{i - 1}(t) < e_1 < e_3 < y_i(t) < y_{i + 1}(t) < e_4 < y_{i + 2}(t), $$
writing $ y_0(t) \coloneqq 0^+ $ and $ y_{s + 1}(t) \coloneqq n^+ $.
We may then define
\begin{equation} \label{eqn:def_retreat34}
    y_i(t+1) \coloneqq e_3^- \quad \text{and} \quad y_j(t+1) \coloneqq y_j(t) \text{ for each } j \in [s] \setminus \{ i \}.
\end{equation}
If there is no $ (3, 4) $-edge in the colouring $ c_t $, then the process terminates with $ t^*_1 \coloneqq t $, and we set $ y_i \coloneqq y_i(t^*_1) $ for each $ i \in [s] $.

Each step of the process may create new monochromatic edges, but it is not hard to see that any new $ (3, 4) $-edges created must appear to the left of $ \vec{e} $, so clearly the process terminates.
Furthermore, we will show that each colouring $ c_t $ remains $ 2 $-well-spread and no $ (1, 2) $-edges are created, which means that any monochromatic edges of $ c_{t^*_1} $ are $ (2, 3) $-edges; we now proceed to handle these in the third stage.

\textbf{Stage 3: Insert switches for $ (2, 3) $-edges.}
Finally, we process any $ (2, 3) $-edges from left to right, inserting a third switch immediately after the third vertex of the edge, ensuring that the third and fourth vertices receive different colours.
Formally, we define a set of additional switches $ Z \subseteq W $ inductively by setting $ Z(0) \coloneqq \emptyset $ and
\begin{equation} \label{eqn:def_close23}
     \quad Z(t + 1) \coloneqq Z(t) \cup \{ e_3^+ \}, \quad \text{where} \quad \vec{e} \text{ is the minimal } (2,3) \text{-edge in } \hat{c}_t \coloneqq c_{Y \cup Z(t)},
\end{equation}
for each $ t \ge 0 $.
If there is no $ (2, 3) $-edge in the colouring $ \hat{c}_t $, then the process terminates with $ t^*_2 \coloneqq t $, and we set $ Z \coloneqq Z(t^*_2) $.
We will show in \cref{lem:inv_23_close}, using properties of $ X $ and $ Y $, that this process creates no new monochromatic edges, and thus terminates in a proper $ 2 $-colouring $ c \coloneqq \hat{c}_{t^*_2} = c_{Y \cup Z} $.

In the next section, we prove \cref{thm:shift_chain_2col} by showing that the algorithm above terminates and yields a proper $ 2 $-colouring. 

\section{Proof of correctness} \label{sect:proof}

To prove that our algorithm terminates in a proper $ 2 $-colouring, we inductively prove various properties satisfied by the collections of switches generated after each stage.
We start by analysing the initial greedy algorithm.

\begin{lemma} \label{lem:inv_greedy}
    Let $ X, x_i, c_0 $ be defined as in \cref{sect:algorithm}.
    The following hold:
    \begin{enumerate}[(X1)]
        \item $ c_0 $ is $2$-well-spread; \label{cond:inv_greedy_2blocks}
        \item if $ \vec{f} $ contains both $ x_i $ and $ x_{i + 1} $ for some $ i \in [s - 1] $, then $ x_{i + 1} = f_4^- $. \label{cond:inv_greedy_2switchend}
    \end{enumerate}
    In particular,
    \begin{enumerate}[(X1)] \setcounter{enumi}{2}
        \item all monochromatic edges in $ c_0 $ are $ (3, 4) $-edges. \label{cond:inv_greedy_mono34}
    \end{enumerate}
\end{lemma}

\begin{proof}
    To prove \ref{cond:inv_greedy_2blocks}, fix $ \vec{f} \in E(\cH) $ and aim to show first that $ \vec{f} $ has at least one switch in $ c_0 $.
    Indeed, let $ i \in [0, s] $ be maximal such that $ x_i < f_4 $.
    If $ x_i < \vec{f} $, then we have $ i \le s - 1 $ by the termination criterion of the algorithm and in particular $ x_{i + 1} = e_4^- $ for some minimal edge $ x_i < \vec{e} \le \vec{f} $ by \eqref{eqn:def_greedy}, which means that $ x_{i + 1} < e_4 \le f_4 $, contradicting the maximality of $ i $.
    As such, we must have $ f_1 < x_i $, which implies that $ i \ge 1 $ and $ \vec{f} $ has at least one switch.
    On the other hand, suppose that $ \vec{f} $ has at least three switches in $ c_0 $.
    Since the set of walls contained in $ \vec{f} $ clearly forms an interval, we must have $ x_{i - 1}, x_i, x_{i + 1} \in S_{c_0}(\vec{f}) $ for some~$ i \in [2, s - 1] $.
    By \eqref{eqn:def_greedy}, we have $ x_{i} = e_4^- $ for some edge $ \vec{e} > x_{i - 1} $.
    It follows that~$ f_1 < x_{i - 1} < e_1 < e_4 < x_{i + 1} < f_4 $, contradicting the definition of a shift-chain.

    For \ref{cond:inv_greedy_2switchend}, suppose that $ \vec{f} $ contains $ x_i $ and $ x_{i + 1} $ for some $ i \in [s - 1] $.
    By definition, we have~$ x_{i + 1} = e_4^- $ for some edge $ \vec{e} > x_i $.
    Since $ f_1 < x_i < e_1 $ and $ \cH $ is a shift-chain, we see that $ \vec{f} < \vec{e} $, and thus~$ e_4^- = x_{i + 1} < f_4 \le e_4 $, so in fact $ f_4 = e_4 $ and $ x_{i + 1} = f_4^- $.

    To see the final conclusion, note by \ref{cond:inv_greedy_2blocks} that all monochromatic edges must have exactly two switches with the same position, that is, they are $ (j, j+1) $-edges for some $ j \in [3] $.
    Suppose an edge $ \vec{f} $ has exactly two switches, which must clearly be $ x_j $ and $ x_{j + 1} $ for some $ j \in [s - 1] $, then by \ref{cond:inv_greedy_2switchend} the switch $ x_{j + 1} = f_4^- $ has position $ (3, 4) $ in $ \vec{f} $.
    Hence, if $ \vec{f} $ is monochromatic, then it must be a $ (3, 4) $-edge.
\end{proof}

\textbf{Simple $ 3 $-colouring algorithm.}
Before analysing the second and third stages of our algorithm, we now outline a simple proof that all $2$-shift-chains, and thus all shift-chains, are properly $ 3 $-colourable, as promised in \cref{sect:shift_chain_obs}.
Let $ \cF $ be a $ 2 $-shift-chain on vertex set $ [n] $ and observe that we may define Stage 1 of our algorithm analogously for $ 2 $-shift-chains by simply replacing $ e_4^-$ with $ e_2^- $ in \eqref{eqn:def_greedy}.
Let $ X \coloneqq \{ x_1, \ldots, x_s \} $ be the resulting set of switches.
Furthermore, by simply replacing every reference to $ e_4 $ or $ f_4 $ in the proof of \ref{cond:inv_greedy_2blocks} by $ e_2 $ or $ f_2 $, respectively, we obtain that every edge of $ \cF $ contains either one or two switches.
Now let $ \overline{c}: [n] \to [3] $ be the colouring obtained by taking $ \overline{c}(1) \coloneqq 1 $ and incrementing the current colour (modulo 3) at every switch; in other words, we define
$$ \overline{c}(1) \coloneqq 1 \quad \text{and} \quad \overline{c}(v) \coloneqq \begin{cases}
    \overline{c}(v - 1),& v^- \not \in X;\\
    \overline{c}(v - 1) + 1 \text{ mod } 3,& v^- \in X.
\end{cases} $$
As such, $ \overline{c} $ is a natural analogue of $ c_X $ with three colours.
By definition, vertices $ u, v \in [n] $ with $ u < v $ receive the same colour if and only if the number of switches $ x \in X $ with $ u < x < v $ is a multiple of 3.
In particular, since each edge $ \vec{e} \in \cF $ contains either one or two switches from $ X $, the two vertices of $ \vec{e} $ must receive different colours, and so $ \overline{c} $ is a proper $ 3 $-colouring, as required.

We now return to the analysis of our three-stage algorithm for $ 4 $-shift-chains, in order to obtain a proper $ 2 $-colouring.
We next prove a set of invariants for the second stage, showing in particular that the only remaining monochromatic edges are $ (2, 3) $-edges.

\begin{lemma} \label{lem:inv_34_retreat}
    Let $ Y(t), y_i(t), c_t, t^*_1 $ be defined as in \cref{sect:algorithm}.
    The following hold for each $ t \in [0, t^*_1] $:
    \begin{enumerate}[(Y1)]
        \item $ c_t $ is $2$-well-spread; \label{cond:inv_retreat_2blocks}
        \item $ x_{j - 1} < y_j(t) \le x_j $ for each $ j \in [s] $; \label{cond:inv_retreat_xyx}
        \item if $ f_1 < x_j $, then $ f_2 < y_{j + 1}(t) $, for any $ \vec{f} \in E(\cH) $ and $ j \in [s] $; \label{cond:inv_retreat_12stop}
        \item there are no $ (1, 2) $-edges in $ c_t $; \label{cond:inv_retreat_monotype}
        \item if $ \vec{f} $ is a $ (2, 3) $-edge in $ c_t $ with switches $ y_{j - 1}(t) $ and $ y_j(t) $ for some $ j \in [2, s] $, then $ f_3^- = y_j(t) < x_j $, and there exists $ \vec{g} \in E(\cH) $ with $ y_j(t) = g_3^- $ and $ y_{j + 1}(t) < g_4 $. \label{cond:inv_retreat_23structure}
    \end{enumerate}
    Furthermore, the process terminates with $ t^*_1 < n $ and
    \begin{enumerate}[(Y1)] \setcounter{enumi}{5}
        \item all monochromatic edges in $ c_Y $ are $ (2, 3) $-edges. \label{cond:inv_retreat_mono23}
    \end{enumerate}
\end{lemma}

\begin{proof}
    We work by induction on $ t \ge 0 $, noting for the base case $ t = 0 $ that \ref{cond:inv_retreat_2blocks} holds trivially by \ref{cond:inv_greedy_2blocks} and \ref{cond:inv_retreat_xyx} holds trivially since $ y_j(0) = x_j $ by \eqref{eqn_def_init_retreat}.
    To see \ref{cond:inv_retreat_12stop}, note that if $ f_1 < x_j $ and $ x_{j + 1} = y_{j + 1}(0) < f_2 $, then by \ref{cond:inv_greedy_2switchend} we have $ f_2 < f_4^-= x_{j + 1} $, a contradiction.
    Finally, both \ref{cond:inv_retreat_monotype} and \ref{cond:inv_retreat_23structure} follow from the fact that there are no $ (1, 2) $- or $ (2, 3) $-edges in $ c_0 $ by \ref{cond:inv_greedy_mono34}.

    For the inductive step, suppose that \ref{cond:inv_retreat_2blocks}--\ref{cond:inv_retreat_23structure} hold for some $ t \ge 0 $ and recall from the definition~\eqref{eqn:def_retreat34} that $ \vec{e} \coloneqq \vec{e}(t) $ is a maximal $ (3, 4) $-edge in the colouring $ c_t $ and $ i \coloneqq i(t) \in [s - 1] $ is such that 
    \begin{equation} \label{eqn:inv_34retreat_pre_order}
        y_{i - 1}(t) < e_1 < e_3 < y_i(t) < y_{i + 1}(t) < e_4 < y_{i+2}(t).
    \end{equation}
    We obtain from \eqref{eqn:def_retreat34} that
    \begin{equation} \label{eqn:inv_34retreat_order}
        y_{i - 1}(t+1) < e_1 < e_2 < y_i(t+1) < e_3 < y_{i + 1}(t+1) < e_4 < y_{i+2}(t+1).
    \end{equation}
    In order to show that \ref{cond:inv_retreat_2blocks}--\ref{cond:inv_retreat_23structure} hold for $ t + 1 $, it will be helpful to observe that
    \begin{equation} \label{eqn:inv_34_retreat_xbound}
        x_{i - 1} < \vec{e}.
    \end{equation}
    Indeed, if $ i = 1 $ then \eqref{eqn:inv_34_retreat_xbound} is trivial, so suppose for contradiction that $ i \ge 2 $ and $ e_1 < x_{i - 1} $.
    We have $ x_i < y_{i + 1}(t) < e_4 $ by \ref{cond:inv_retreat_xyx} and \eqref{eqn:inv_34retreat_pre_order}, which means that $ e_1 < x_{i - 1} < x_i < e_4 $, so by \ref{cond:inv_greedy_2switchend} we obtain $ x_i = e_4^- $, a contradiction.

    To prove \ref{cond:inv_retreat_2blocks}, fix an edge $ \vec{f} \in E(\cH) $ and firstly suppose for contradiction that $ \vec{f} $ has no switches in $ c_{t + 1} $.
    This means that $ \vec{f} $ does not contain $ y_j(t + 1) = y_j(t) $ for any $ j \ne i $, but by the inductive hypothesis $ \vec{f} $ has at least one switch in $ c_t $, so we must have $ S_{c_t}(\vec{f}) = \{ y_i(t) \} $ and $ y_i(t+1) < \vec{f} $.
    Using these two facts and \eqref{eqn:inv_34retreat_order}, we obtain that
    $$ e_1 < y_i(t+1) < f_1 < y_i(t) < f_4 < y_{i + 1}(t) < e_4, $$
    contradicting the definition of a shift-chain.
    Now suppose instead that $ \vec{f} $ has at least three switches in $ c_{t + 1} $.
    By the inductive hypothesis, $ \vec{f} $ has at most two switches in $ c_t $, and since $ y_i $ is the only switch which changes, this implies that $ \vec{f} $ contains $ y_i(t+1) $ but not $ y_i(t) $.
    In particular, since the walls contained by $ \vec{f} $ form an interval and $ y_i(t+1) < y_i(t) < y_{i + 1}(t) = y_{i + 1}(t + 1) $ by definition, we must have $ i \ge 3 $ and
    $$ S_{c_{t+1}}(\vec{f}) = \{ y_{i - 2}(t+1), y_{i - 1}(t+1), y_i(t+1) \} = \{ y_{i - 2}(t), y_{i - 1}(t), e_3^- \}. $$
    Thus we obtain $ f_1 < y_{i - 2}(t) \le x_{i - 2} $ by \ref{cond:inv_retreat_xyx} and
    \begin{equation} \label{eqn:inv_34_retreat_y1ineq}
        x_{i - 1} < e_1 < e_3^- < f_4
    \end{equation}
    by \eqref{eqn:inv_34_retreat_xbound}, which means that $ \vec{f} $ contains both of $ x_{i - 2} $ and $ x_{i - 1} $, so by \ref{cond:inv_greedy_2switchend} we have $ x_{i - 1} = f_4^- $, contradicting \eqref{eqn:inv_34_retreat_y1ineq}.

    To prove \ref{cond:inv_retreat_xyx}, it suffices to prove the statement for $ j = i $ by the inductive hypothesis, since all other switches $ y_j $ remain unchanged.
    In this case, by \eqref{eqn:inv_34_retreat_xbound}, \eqref{eqn:inv_34retreat_order}, and the inductive hypothesis, we have $ x_{i - 1} < e_1 < y_i(t+1) < y_i(t) \le x_i $, as required.
    
    Similarly, for \ref{cond:inv_retreat_12stop}, it suffices to prove the case $ j = i - 1 $, as otherwise $ y_{j + 1} $ remains unchanged.
    Let $ \vec{f} \in E(\cH) $ satisfy $ f_1 < x_{i - 1} $, then we have $ f_1 < x_{i - 1} < e_1 $ by \eqref{eqn:inv_34_retreat_xbound}, which means that $ \vec{f} < \vec{e} $ as $ \cH $ is a shift-chain, so in particular $ f_2 \le e_2 < y_i(t+1) $ by \eqref{eqn:inv_34retreat_order}, as required.

    We prove \ref{cond:inv_retreat_monotype} by contradiction, so suppose that $ \vec{f} $ is a $ (1, 2) $-edge in $ c_{t + 1} $, which means by definition that
    $$ y_{j - 1}(t+1) < f_1 < y_j(t+1) < y_{j+1}(t+1) < f_2 < f_4 < y_{j + 2}(t+1) $$
    for some $ j \in [s - 1] $.
    If $ j \ne i - 1 $, then by \eqref{eqn:def_retreat34}, we have
    $$ f_1 < y_j(t+1) \le y_j(t) < y_{j+1}(t) = y_{j+1}(t+1) < f_2, $$
    which means that $ \vec{f} $ has switches $ y_j(t) $ and $ y_{j + 1}(t) $ with position $ (1, 2) $ in $ c_t $, and by \ref{cond:inv_retreat_2blocks} these must be its only switches in $ c_t $, so we see that $ \vec{f} $ is a $ (1, 2) $-edge in $ c_t $, contradicting the inductive hypothesis.
    If instead $ j = i - 1 $, then using \eqref{eqn:inv_34retreat_order} we obtain
    $$ f_1 < y_{i - 1}(t + 1) < e_1 < e_2 < y_i(t+1) < f_2, $$
    so $ \vec{e} $ and $ \vec{f} $ are not comparable, contradicting the definition of a shift-chain.

    For \ref{cond:inv_retreat_23structure}, fix some $ (2, 3) $-edge $ \vec{f} \in E(\cH) $ in $ c_{t + 1} $.
    Suppose first that $ \vec{f} $ does not contain $ y_i(t+1) $, then since all other switches are unchanged and $ \vec{f} $ has at most two switches in $ c_t $ by~\ref{cond:inv_retreat_2blocks}, we see that $ \vec{f} $ is also a $ (2, 3) $-edge in $ c_t $ with the same switch set as in $ c_{t + 1} $.
    In this case, writing
    $$ S_{c_{t+1}}(\vec{f}) = S_{c_t}(\vec{f}) = \{ y_{j - 1}(t), y_j(t) \} $$
    for some $ j \in [2, s] \setminus \{ i \} $, we have by the inductive hypothesis and definitions that $ f_3^- = y_j(t) = y_j(t+1) < x_j $, and there exists $ \vec{g} \in E(\cH) $ with $ y_j(t) = y_j(t+1) = g_3^- $ and $ y_{j + 1}(t + 1) \le y_{j + 1}(t) < g_4 $, as required.
    Now suppose instead that $ y_i(t+1) \in S_{c_{t+1}}(\vec{f}) $, which means that the second switch of $ \vec{f} $ in $ c_{t + 1} $ is either $ y_{i - 1}(t+1) $ or $ y_{i+1}(t+1) $, since the walls contained by $ \vec{f} $ form an interval.
    If
    $$ S_{c_{t+1}}(\vec{f}) = \{ y_{i - 1}(t+1), y_i(t+1) \}, $$
    then by assumption and \eqref{eqn:inv_34retreat_order}, we have $ f_4 < y_{i + 1}(t+1) < e_4 $, which means that $ \vec{f} < \vec{e} $ and in particular~$ f_3 \le e_3 $.
    However, since $ e_3^- = y_i(t+1) < f_3 $ by \eqref{eqn:def_retreat34} and our assumption that $ \vec{f} $ is a $ (2, 3) $-edge, we see that also $ e_3 \le f_3 $ and thus $ f_3 = e_3 $.
    Since also $ y_i(t+1) < y_i(t) \le x_i $ by definition and \ref{cond:inv_retreat_xyx} and $ y_{i + 1}(t + 1) < e_4 $ by \eqref{eqn:inv_34retreat_order}, this suffices for \ref{cond:inv_retreat_23structure}, with $ \vec{e} $ playing the role of $ \vec{g} $.
    If instead
    $$ S_{c_{t+1}}(\vec{f}) = \{ y_i(t+1), y_{i + 1}(t+1) \}, $$
    then by definition we have
    $$ f_2 < y_i(t+1) < y_i(t) < y_{i+1}(t) = y_{i + 1}(t+1) < f_3. $$
    In other words, the edge $ \vec{f} $ has switches $ y_i(t) $ and $ y_{i + 1}(t) $ in $ c_t $, so by \ref{cond:inv_retreat_2blocks} it follows that these are the only such switches and $ \vec{f} $ is a $ (2, 3) $-edge in $ c_t $.
    Hence by the inductive hypothesis we obtain that $ f_3^- = y_{i+1}(t) = y_{i+1}(t+1) < x_{i + 1} $ and there exists $ \vec{g} \in E(\cH) $ with $ y_{i+1}(t+1) = y_{i +1}(t) = g_3^- $ and $ y_{i + 2}(t + 1) = y_{i + 2}(t) < g_4 $, as required.
    This completes the inductive step.
    
    Finally, to see that $ t^*_1 < n $, we first show that $ i(t) $ is non-increasing, for which it suffices to show that $ \vec{e}(t) $ is (strictly) decreasing, since if $ i(t + 1) \ge i(t) + 1 $, then
    $$ e_4(t) < y_{i(t) + 2}(t + 1) \le y_{i(t + 1) + 1}(t + 1) < e_4(t + 1) $$
    by \eqref{eqn:inv_34retreat_order} and \eqref{eqn:inv_34retreat_pre_order}.
    To see that $ \vec{e}(t) $ is decreasing, observe that if $ \vec{f} \coloneqq \vec{e}(t + 1) > \vec{e} \coloneqq \vec{e}(t) $, then by maximality of $ \vec{e} $, it must be the case that $ \vec{f} $ is a $ (3, 4) $-edge in $ c_{t + 1} $ but not $ c_t $.
    Then by \eqref{eqn:def_retreat34}, we have $ y_i(t + 1) = e_3^- < e_3 \le f_3 $, so since $ \vec{f} $ is a $ (3, 4) $-edge, it follows that $ \vec{f} > y_i(t+1) $.
    If $ \vec{f} $ does not contain $ y_i(t) $, then the switches of $ \vec{f} $ do not change, so $ \vec{f} $ must be a $ (3, 4) $-edge in $ c_t $, a contradiction.
    If instead $ \vec{f} $ contains $ y_i(t) $, then $ \vec{f} $ has at most one other switch in $ c_t $ by \ref{cond:inv_retreat_2blocks}.
    Since $ \vec{f} $ does not contain $ y_i(t+1) $ and all other switches remain unchanged, this means that $ \vec{f} $ has at most one switch in $ c_{t + 1} $, contradicting it being a $ (3, 4) $-edge.
    
    Now since $ i(t) $ is non-increasing, it follows in particular that $ y_{i(t)}(t + 1) < y_{i(t)}(t) \le y_{i(t-1)}(t) $ for $ t \in [t^*_1 - 1] $, so the sequence $ (y_{i(t - 1)}(t))_{t \in [t^*_1]} $ is strictly decreasing with $ 0^+ < y_{i(t^*_1 - 1)}(t^*_1) < y_{i(0)}(1) < n^+ $, implying that $ t^*_1 < n $.
    Since the process may only terminate when no $ (3, 4) $-edges remain, it follows by \ref{cond:inv_retreat_2blocks} and \ref{cond:inv_retreat_monotype} that every monochromatic edge in $ c_Y $ is a $ (2, 3) $-edge, as required for \ref{cond:inv_retreat_mono23}.
\end{proof}

Finally, we prove a further set of invariants for the third stage, showing in particular that no new monochromatic edges are created.
 
\begin{lemma} \label{lem:inv_23_close}
    Let $ Y, Z(t), \hat{c}_t, t^*_2 $ be defined as in \cref{sect:algorithm}.
    The following hold for each $ t \in [0, t^*_2] $:
    \begin{enumerate}[(Z1)]
        \item $ \hat{c}_t $ is $ 3 $-well-spread; \label{cond:inv_close_3blocks}
        \item $ M(\hat{c}_t) \subsetneq M(\hat{c}_{t - 1}) $ if $ t \ge 1 $; \label{cond:inv_close_mono_fewer}
        \item every monochromatic edge $ \vec{f} $ in $ \hat{c}_t $ is a $ (2, 3) $-edge with $ z < \vec{f} $ for all $ z \in Z(t) $; \label{cond:inv_close_mono23y}
        \item for each $ z \in Z(t) $, there exist $ j \in [s] $ and $ v \in [n] $ such that $ y_j = v^- $ and $ z = v^+ \le x_j $. \label{cond:inv_close_new_structure} 
    \end{enumerate}
    In particular, the process terminates with $ t^*_2 < n $ and $ c_{Y \cup Z} $ is a proper $ 2 $-colouring.
\end{lemma}

\begin{proof}
    Recall the definition \eqref{eqn:def_close23}.
    We work by induction on $ t \ge 0 $, noting for the base case $ t = 0 $ that \ref{cond:inv_close_3blocks} and \ref{cond:inv_close_mono23y} hold trivially by \ref{cond:inv_retreat_2blocks} and \ref{cond:inv_retreat_mono23}, respectively, and \ref{cond:inv_close_mono_fewer} and \ref{cond:inv_close_new_structure} are vacuous.
    For the inductive step, suppose that \ref{cond:inv_close_3blocks}--\ref{cond:inv_close_new_structure} hold for some $ t \ge 0 $, recall that $ \vec{e} = \vec{e}(t) $ is a minimal $ (2, 3) $-edge in the colouring $ \hat{c}_t $, and write $ \hat{z} \coloneqq e_3^+ $.
    By \ref{cond:inv_close_mono23y}, the two switches of $ \vec{e} $ in $ \hat{c}_t $ must both belong to $ Y $, so by definition there exists $ i \in [2, s] $ such that
    \begin{equation} \label{eqn:inv_close_order}
        y_{i - 2} < e_1 < e_2 < y_{i - 1} < y_i < e_3 < \hat{z} < e_4 < y_{i + 1}
    \end{equation}
    and none of the new switches in $ Z(t) $ are contained by $ \vec{e} $.
    Clearly $ \vec{e} $ is not monochromatic in the colouring $ \hat{c}_{t+1} $, since $ e_3 $ and $ e_4 $ receive different colours.
    Also, since $ \vec{e} $ is clearly a $ (2, 3) $-edge in~$ \hat{c}_0 $, it follows from \ref{cond:inv_retreat_23structure} that
    \begin{equation} \label{eqn:inv_close_ret23struct}
        y_i = e_3^- < x_i \quad \text{so} \quad e_3^+ \le x_i, \quad \text{and} \quad g_3 = e_3 < y_{i + 1} < g_4
    \end{equation}
    for some edge $ \vec{g} \in E(\cH) $.
    We now aim to show that \ref{cond:inv_close_3blocks}--\ref{cond:inv_close_new_structure} hold for $ t + 1 $.

    We begin with \ref{cond:inv_close_3blocks}.
    Since we only add the switch $ \hat{z} $ and do not change any existing switches, it is clear that every edge has at least one switch in $ \hat{c}_{t + 1} $ by the inductive hypothesis.
    On the other hand, suppose for contradiction that some edge $ \vec{f} $ has at least four switches in $ \hat{c}_{t + 1} $, noting that $ \hat{z} $ must be one of these since $ \vec{f} $ has at most three switches in $ \hat{c}_t $ by the inductive hypothesis.
    Furthermore, $ \vec{f} $ has at most two switches in $ c_Y $ by \ref{cond:inv_retreat_2blocks}, so there must also be a second new switch $ z \in Z(t) \cap S_{\hat{c}_t}(\vec{f}) $.
    By \ref{cond:inv_close_new_structure}, there exist $ j \in [s] $ and a vertex $ v \in [n] $ such that $ y_j = v^- < v^+ = z \le x_j $.
    By \ref{cond:inv_close_mono23y} and \eqref{eqn:inv_close_order}, we have $ z < e_1 < y_{i - 1} $, so $ y_j < y_{i - 1} $ and we deduce that $ j \le i - 2 $ and in particular $ f_1 < z \le x_j \le x_{i - 2} $, since by assumption $ z $ is contained by $ \vec{f} $.
    However, we also obtain that $ x_{i - 1} < y_i $ by \ref{cond:inv_retreat_xyx}, as well as $ y_i < \hat{z} $ by \eqref{eqn:inv_close_order} and $ \hat{z} < f_4 $ by the assumption that~$ \vec{f} $ contains $ \hat{z} $.
    We conclude that $ f_1 < x_{i - 2} < x_{i - 1} < f_4^- $, so in particular $ \vec{f} $ contains both of the walls $ x_{i - 2} $ and $ x_{i - 1} $ but $ x_{i - 1} < f_4^- $, contradicting \ref{cond:inv_greedy_2switchend}.

    To prove \ref{cond:inv_close_mono_fewer}, it suffices to show that no new monochromatic edge is created in $ \hat{c}_{t + 1} $ which was not present in $ \hat{c}_t $, since $ \vec{e} $ is monochromatic in $ \hat{c}_t $ but not $ \hat{c}_{t + 1} $.
    Suppose for contradiction that $ \vec{f} $ is monochromatic in $ \hat{c}_{t + 1} $ but not $ \hat{c}_t $.
    By \ref{cond:inv_close_3blocks} for $ \hat{c}_{t + 1} $,
    the edge $ \vec{f} $ must have exactly two switches in $ \hat{c}_{t + 1} $, both with the same position.
    The new switch $ \hat{z} $ is clearly one of these, since $ \vec{f} $ is not monochromatic in $ \hat{c}_t $ and no existing switches are changed, so the second switch must be either $ y_i $ or $ y_{i + 1} $, since $ z < y_i < \hat{z} < y_{i + 1} $ for every $ z \in Z(t) $ by \ref{cond:inv_close_mono23y} and~\eqref{eqn:inv_close_order}.
    We consider six possible cases, corresponding to the two possible switch sets and three possible positions in~$ \vec{f} $.
    Suppose first that $ \vec{f} $ has switches $ y_i $ and $ \hat{z} $ in $ \hat{c}_{t + 1} $.
    If $ f_3 < y_i < \hat{z} < f_4 $, then $ f_3 < y_i < e_3 $ by \eqref{eqn:inv_close_order}, so $ \vec{f} < \vec{e} $ and in particular $ f_1 \le e_1 < y_{i - 1} $ by \eqref{eqn:inv_close_order}, but since $ y_{i - 1} < y_i < f_4 $, this means that $ \vec{f} $ also contains the switch $ y_{i - 1} $, a contradiction.
    If $ f_2 < y_i < \hat{z} < f_3 $ or $ f_1 < y_i < \hat{z} < f_2 $, then, recalling the definition of $ \vec{g} $ from \eqref{eqn:inv_close_ret23struct} and the fact that $ \hat{z} = e_3^+ = g_3^+ $, we obtain $ g_3 < \hat{z} < f_3 $, so $ \vec{g} < \vec{f} $ and in particular $ y_{i + 1} < g_4 \le f_4 $ by \eqref{eqn:inv_close_ret23struct}, which means that $ \vec{f} $ also contains the switch $ y_{i + 1} $, a contradiction.
    Now suppose instead that $ \vec{f} $ has switches $ \hat{z} $ and $ y_{i + 1} $ in $ \hat{c}_{t + 1} $.
    Then $ y_i < f_1 < \hat{z} $, and since also $ y_i = e_3^- < e_3^+ = \hat{z} $ by \eqref{eqn:inv_close_ret23struct}, it follows that $ f_1 = e_3 $ and $ \hat{z} = f_1^+ < f_2 $.
    Recalling that the two switches of $ \vec{f} $ have the same position, this means that the only possible position for the switches $ \hat{z} $ and $ y_{i + 1} $ in $ \vec{f} $ is~$ (1, 2) $, and in particular $ y_{i + 1} < f_2 $.
    Since $ f_1 < \hat{z} = e_3^+ \le x_i $ by \eqref{eqn:inv_close_ret23struct}, this contradicts \ref{cond:inv_retreat_12stop}.
    This concludes the case check and thus the proof of \ref{cond:inv_close_mono_fewer}.

    To prove \ref{cond:inv_close_mono23y}, suppose that $ \vec{f} $ is a monochromatic edge in $ \hat{c}_{t + 1} $, in which case we have already shown that $ \vec{f} $ must also be monochromatic in $ \hat{c}_t $, and thus a $ (2, 3) $-edge in $ \hat{c}_t $ by the inductive hypothesis.
    We have also shown already that $ \vec{f} $ has at most three switches in $ \hat{c}_{t + 1} $, which means that to be monochromatic it must have exactly two switches in $ \hat{c}_{t + 1} $.
    Since no existing switches are removed, it follows that $ S_{\hat{c}_{t+1}}(\vec{f}) = S_{\hat{c}_t}(\vec{f}) $, so $ \vec{f} $ does not contain the new switch $ \hat{z} $ and in particular is also a $ (2, 3) $-edge in $ \hat{c}_{t + 1} $.
    Using the fact that $ \vec{e} < \vec{f} $ by minimality and thus $ \hat{z} < e_4 \le f_4 $ by \eqref{eqn:inv_close_order}, it follows that $ \hat{z} < f_1 $.
    This is sufficient for \ref{cond:inv_close_mono23y}, since by the inductive hypothesis $ z < f_1 $ for every $ z \in Z(t) $.

    For \ref{cond:inv_close_new_structure}, it suffices by the inductive hypothesis to prove that the statement holds for the new switch $ \hat{z} $.
    Indeed, this is immediate by \eqref{eqn:inv_close_ret23struct}, with $ (i, e_3) $ playing the role of $ (j, v) $.
    This completes the inductive step.

    Since the new switches $ \hat{z} $ added at each step are distinct by definition, and there are at most $ n - 1 $ possible walls which could be added, it is clear that the process terminates with $ t^*_2 < n $.
    Furthermore, by \ref{cond:inv_close_mono23y} and the termination condition, the colouring $ \hat{c}_{t^*_2} $ induces no monochromatic edges, as required.
\end{proof}

This completes the proof of \cref{thm:shift_chain_2col}.
We conclude by briefly justifying that the algorithm we have described can be implemented to run in linear time.

\subsection{Linear time complexity} \label{sect:linear}

Since each edge has a unique sum of its elements in the range $ [10, 4n - 6] $, it is easy to see that there are at most linearly many edges (as discussed in~\cref{sect:shift_chain_obs}) and in particular that we may sort the list of edges in linear time.
The first stage can then be implemented in amortised linear time, by iterating over the vertices and maintaining a pointer to the minimal edge which has not yet been inspected.
For the second stage, we may iterate backwards over the list of edges and track the index $ i(t) $, with which we can easily check in constant time whether a given edge is a $ (3, 4) $-edge.
Since the index $ i(t) $ decreases and new monochromatic edges always appear to the left, it is clear that this requires only linear time.
The same is true for the third stage, since we may iterate forwards over the list of edges and process any $ (2, 3) $-edges found in constant time.

\bibliographystyle{amsplain}
\bibliography{sources}

\end{document}